\documentclass[11pt,a4paper]{article}
\usepackage[left=3cm,right=3cm,top=3cm,bottom=3cm]{geometry}
\usepackage{xcolor}
\usepackage[noend]{algpseudocode}
\usepackage[ruled]{algorithm2e}
\usepackage{comment}
\usepackage{cite}
\usepackage{amssymb}
\usepackage{amsmath}
\usepackage{amsthm}
\usepackage{amsfonts}
\usepackage{graphicx}
\usepackage{enumerate}
\usepackage[ansinew]{inputenc}
\usepackage{hyperref}
\usepackage{multirow}
\usepackage{color}
\usepackage{colortbl}   

\usepackage{wrapfig}
\usepackage{floatflt,epsfig} 
\usepackage{fancyhdr}
\usepackage{array}
\usepackage[right]{eurosym}

\newtheorem{theorem}{Theorem}
\newtheorem{def.}{Definition}
\newtheorem{prop.}[theorem]{Proposition}
\newtheorem{lem.}[theorem]{Lemma}
\newtheorem{cor.}[theorem]{Corollary}
\newtheorem{conj.}[theorem]{Conjecture}
\newtheorem{example}{Example}

\newtheorem{rem.}[theorem]{Remark}

\newcommand{\norm}[2]{
\left\| #2 \right\|_{#1}
}

\def\be{\begin{equation}}
\def\ee{\end{equation}}
\def\ben{\begin{eqnarray}}
\def\een{\end{eqnarray}}

\newcommand{\range}[1]{\mathsf{ran}\left( #1 \right)} %sf rm tt
\newcommand{\mydiag}[1]{\mathsf{diag}\left( #1 \right)} %sf rm tt
\newcommand{\mydist}[1]{\mathsf{dist}\left( #1 \right)} %sf rm tt
\newcommand{\support}[1]{\mathsf{supp}\left( #1 \right)} %sf rm tt
\newcommand{\identity}[1]{\mathsf{id}_{ #1 }}
\newcommand{\kernel}[1]{\mathsf{ker}\left( #1 \right)}

\newcommand{\finspan}[1]{\mathsf{span}\left\{ #1 \right\}}

\definecolor{darkviolet}{rgb}{0.58,0,0.83} %{148,0,211}

\def\letters{a,b,c,d,e,f,g,h,i,j,k,l,m,n,o,p,q,r,s,t,u,v,w,x,y,z}
\def\Letters{A,B,C,D,E,F,G,H,I,J,K,L,M,N,O,P,Q,R,S,T,U,V,W,X,Y,Z}
\makeatletter
\@for \@l:=\Letters \do{%
  \expandafter\edef\csname\@l bb\endcsname{%
  \noexpand\ensuremath{\noexpand\mathbb{\@l}}}%
  \expandafter\edef\csname\@l bf\endcsname{{\noexpand\bf \@l}}%
  \expandafter\edef\csname\@l cal\endcsname{%
  \noexpand\ensuremath{\noexpand\mathcal{\@l}}}%
  \expandafter\edef\csname\@l eu\endcsname{%
  \noexpand\ensuremath{\noexpand\EuScript{\@l}}}%
  \expandafter\edef\csname\@l frak\endcsname{%
  \noexpand\ensuremath{\noexpand\mathfrak{\@l}}}%
  \expandafter\edef\csname\@l rm\endcsname{{\noexpand\rm \@l}}%
  \expandafter\edef\csname\@l scr\endcsname{%
  \noexpand\ensuremath{\noexpand\mathscr{\@l}}}%
}
\@for \@l:=\letters \do{%
  \expandafter\edef\csname\@l bf\endcsname{{\noexpand\bf \@l}}%
  \expandafter\edef\csname\@l frak\endcsname{%
  \noexpand\ensuremath{\noexpand\mathfrak{\@l}}}%
  \expandafter\edef\csname\@l scr\endcsname{%
  \noexpand\ensuremath{\noexpand\mathscr{\@l}}}%
}
\makeatother
\newcommand{\isdef}{\mathrel{\mathrel{\mathop:}=}}
\newcommand{\defis}{\mathrel{=\mathrel{\mathop:}}}
\newcommand{\bs}{\boldsymbol}

\renewcommand{\d}{\operatorname{d}\!}
\DeclareMathOperator{\diam}{diam}

\newtheorem{remark}[theorem]{Remark}
\newcommand{\vertiii}[1]{{\left\vert\kern-0.25ex\left\vert\kern-0.25ex\left\vert #1 
    \right\vert\kern-0.25ex\right\vert\kern-0.25ex\right\vert}}
 
\title{Construction of generalized samplets in Banach spaces}
\author{Peter Balazs and Michael Multerer}
\date{\today}

\begin{document}

\maketitle

\begin{abstract}
Recently, samplets have been introduced as localized discrete signed
measures which are tailored to an underlying data set. Samplets exhibit
vanishing moments, i.e., their measure integrals vanish for all polynomials up 
to a certain degree, which allows for feature detection and data compression.
In the present article, we extend the different construction steps of samplets
to functionals in Banach spaces more general than point 
evaluations. To obtain stable representations, we assume that these functionals 
form frames with square-summable coefficients or even Riesz bases with 
square-summable coefficients. In either case, the corresponding analysis 
operator is injective and we obtain samplet bases with the desired properties by
means of constructing an isometry of the analysis operator's image. Making the
assumption that the dual of the Banach space under consideration is imbedded
into the space of compactly supported distributions, the multilevel
hierarchy for the generalized samplet construction is obtained by spectral 
clustering of a similarity graph for the functionals' supports. Based on this
multilevel hierarchy, generalized samplets exhibit vanishing moments 
with respect to a given set of primitives within the Banach space.
We derive an abstract localization result for the generalized samplet 
coefficients with respect to the samplets' support sizes and the approximability
of the Banach space elements by the chosen primitives. Finally, we present three
examples showcasing the generalized samplet framework. 
\end{abstract}

%%%%%%%%%%%%%%%%%%%%%%%%%%%%%%%%%%%%%%%%%%%%%%%%%%%%%%%%%%%%%%%%%%%%%%%%%%%%%%%%
\section{Introduction}
%%%%%%%%%%%%%%%%%%%%%%%%%%%%%%%%%%%%%%%%%%%%%%%%%%%%%%%%%%%%%%%%%%%%%%%%%%%%%%%%
Unstructured data is ubiquitous, while the amount of data is immense and
rapidly increasing. The processing of social network data, text data,
audio files, photos and videos, but also scientific data, like measurements
and simulation data has become vital to our modern society. Multiresolution
methods in general and wavelets in particular are well-established tools
for nonlinear approximation, image analysis, signal processing, 
and machine learning, see, e.g., \cite{HM_Chui,DA,HM_Daubechies,HM_Mallat}.
Even so, wavelets are typically restricted to structured data, such as uniform
subdivisions of the real line or bounded intervals, 
see \cite{HM_Alp93,HM_Quak,HM_DKU}, with higher dimensional constructions
relying on tensorization, see \cite{HM_STE,HM_HS,HM_PSS97}. Wavelet-like
multiresolution analyses for unstructured data have been suggested in 
\cite{HM_CM06,HM_RE11,GNC10,HM22,TW03} and rely on point evaluations or
local averages. In particular, the construction of Tausch-White 
multi-wavelet bases in \cite{TW03} has recently been transferred 
to discrete signed measures in \cite{HM22} resulting to the concept of 
\emph{samplets}. Samplets are tailored to the underlying data set and 
can be constructed such that they exhibit \emph{vanishing moments}, i.e.,
their measure integrals vanish for all polynomials up to a certain degree.

This article focuses on extending the construction methodology of samplets to
accommodate a broader class of functionals in Banach spaces beyond point
evaluations. To ensure stable representations, we assume these functionals form
either frames or Riesz bases with square-summable coefficients. 
In both cases, the associated analysis operator is injective, enabling the
construction of samplet bases with the desired properties by designing an 
isometry on the image of the analysis operator.

We remark that the foundational definition of frames, introduced in the 
pioneering work \cite{duffschaef1}, involves a sequence of elements within a 
Hilbert space. Calculating the inner product of an element there with this 
sequence produces the analysis operator that is considered to be both bounded
and boundedly invertible as a map from that Hilbert into the space of 
square-summable sequences. Frames offer significantly more flexibility than 
(orthonormal) bases. In situations, where orthonormal bases are too constrictive,
but one desires to keep some independence characteristics, Riesz bases serve as
a viable alternative, ensuring that the analysis operator remains a bijection.
The frame concept has inspired considerable theoretical exploration, as outlined
in \cite{ole1}, and has been applied extensively in signal processing 
\cite{HM_Mallat}, acoustics \cite{framepsycho16}, as well as many other areas. 
Frames can also be generalized, e.g., to Banach spaces \cite{cashanlar99}. 
This methodology has been utilized to represent not only elements in a
Hilbert space but also operators, using the Galerkin method, see, e.g., 
\cite{sauter2010boundary}. Within the Galerkin method, bases
have traditionally been used \cite{DahSch99a}, but more recently frames have
been applied \cite{stevenson03}. Riesz bases are also used here to preserve some
independence attributes \cite{DA}. Furthermore, frames and Riesz bases have been
studied from a frame-theoretic viewpoint with respect to operator 
representations \cite{xxlframoper1}, in particular also in the localized frame 
setting \cite{xxlgro14}.
Noting that the frame condition in the initial definition of \cite{stevenson03} 
differentiated between a Hilbert space and its dual, \cite{xxlhar18} developed
the theory of Banach frames with square-summable coefficients, demonstrating
that when the specific values of the frame bounds are significant, such an 
approach is warranted. 

In \cite{HM22}, the multilevel hierarchy for the construction of 
samplets has been based on a Euclidean hierarchical clustering of the 
underlying Dirac-\(\delta\)-distributions' supports. For more general
functionals, this approach is not possible anymore. Making the assumption that
the dual space of the Banach space under consideration is imbedded into the
compactly supported distributions and introducing a similarity measure based on
the distributions' supports, we propose to construct the multilevel hierarchy by
means of spectral clustering, see \cite{vLux07,vLBB08}. Afterwards,
generalized samplets are constructed by computing filter coefficients from the
QR decomposition of a cluster's moment matrix with respect to a given set of
primitives, similarly to \cite{AHK14,HM22}. Within this setting, we derive an
abstract localization, i.e., decay result, for the generalized samplet
coefficients with respect to the samplets' support sizes and the approximability
of the Banach space elements by the chosen primitives.

The remainder of this article is structured as follows. In 
Section~\ref{sec:Preliminaries}, we recall general Banach frame and 
Banach Riesz basis related results, which are relevant for the generalized 
samplet construction. In particular, we discuss transformations of frames and 
bases by means of isometries acting on the analysis operator's image.
In Section~\ref{sec:GenFinSamp}, we apply these results in 
case of finite sets functionals, which are the starting point for constructing
generalized samplets. Section~\ref{sct:construction} constitutes the main
contribution of this article. Here, we construct samplets exhibiting vanishing
moments with respect to a given set of primitives by constructing a suitable
isometry on the finite analysis operator's image. In particular, we prove the
abstract localization result for the samplet coefficients in general Banach
spaces. As a corollary, we obtain the well known corresponding result for
wavelets with vanishing moments in Sobolev spaces. Finally, 
in Section~\ref{sec:exampfin0}, we present three examples showcasing the
generalized samplet framework. These examples consider samplets in
reproducing kernel Hilbert spaces, the Tausch-White wavelet construction for
piecewise polynomial ansatz functions and operator adapted wavelets.

%%%%%%%%%%%%%%%%%%%%%%%%%%%%%%%%%%%%%%%%%%%%%%%%%%%%%%%%%%%%%%%%%%%%%%%%%%%%%%%%
\section{Preliminaries}\label{sec:Preliminaries}
%%%%%%%%%%%%%%%%%%%%%%%%%%%%%%%%%%%%%%%%%%%%%%%%%%%%%%%%%%%%%%%%%%%%%%%%%%%%%%%%
\subsection{Banach frames with square-summable coefficients}
We start from the concept of \emph{frames} in Banach spaces, see, e.g.,
\cite{xxlhar18,chst03}.
Throughout this article, let $(\Bcal,\|\cdot\|_\Bcal)$ be a Banach space
having an \emph{\(\ell^2\)-frame}
\({\bs f}\subset\Bcal'\), where \(\Bcal'\) is the
topological dual space of \(\Bcal\).
We denote the duality pairing 
by $(f,v)_{\Bcal'\times\Bcal}$ for any $v\in\Bcal,f\in\Bcal'$. 

\begin{def.}\label{def:l2frame}
 A sequence \( {\bs f} \isdef \{f_i\}_{i\in I}\subset\Bcal'\), \(I\subset\Nbb\), 
 is an 
\emph{$\ell^2$-frame} for $\Bcal$,
if there exist \(A,B>0\) such that
\begin{equation}\label{eq:frame2}
A\|v\|_{\Bcal}^2\leq\big\|
\big\{(f_i,v)_{\Bcal'\times\Bcal}\big\}_{i \in I}\big\|_{\ell^2}^2
\leq B\|v\|_{\Bcal}^2\quad\text{for all }v\in\Bcal.
\end{equation}
We call the sequence \(\bs f\subset\Bcal'\) a 
\emph{Bessel sequence} if the upper inequality is satisfied.
\end{def.}

We remark that it has been shown in \cite[Lemma 2.1]{chst03} that the existence
of a frame
for \(\Bcal\) with coefficients in \(\ell^p(I)\), \(1<p<\infty\), already yields
the reflexivity and the separability of \(\Bcal\).
Definition~\ref{def:l2frame} particularly considers Banach frames with
coefficient sequences in $\ell^2(I)$. Such frames lead to a 
Hilbert space topology, see  \cite{stoev09}, with equivalent but not
identical norms. This difference is considered essential for typical settings in
numerical analysis \cite{xxlhar18}. This is also the (obvious) dual construction
to what was termed \emph{Stevenson frames} in \cite{xxlhar18}, where the frame
is considered to be in $\Bcal$, and the elements $v$ in $\Bcal'$. To make the 
notation shorter we call such a sequence just a \emph{frame}. 

We introduce the \emph{analysis operator}
\begin{equation}\label{eq:analysisOp}
T^\star\colon\Bcal\to\ell^2(I),\quad T^\star v\isdef
\big\{(f_i,v)_{\Bcal'\times\Bcal}\big\}_{i\in I}
\end{equation}
and the \emph{synthesis operator}
\begin{equation}\label{eq:synthesisOp}
T\colon\ell^2(I)\to\Bcal',\quad T{\bs\alpha}\isdef
\sum_{i\in I}\alpha_i f_i.
\end{equation}
Note that \(T\) and \(T^\star\) are indeed adjoint to each other, 
cp.\ \cite{xxlhar18}. 

For Banach frames with arbitrary coefficients spaces, e.g., other 
$\ell^p$-spaces, the combination of analysis and synthesis operators is not
straightforward, see \cite{Stoeva08} for a detailed discussion.
In the present setting, however, there
holds $\ell^2(I) \simeq[{\ell^2}(I)]'$ and we may introduce the
\emph{frame operator}
\begin{equation}\label{eq:frameOp}
S\isdef TT^\star\colon\Bcal\to\Bcal',\quad Sv
=\sum_{i\in I}(f_i,v)_{\Bcal'\times\Bcal}f_i.
\end{equation}
The frame operator is symmetric, i.e.,
\begin{equation}\label{eq:FrameOpSym}
(Su,v)_{\Bcal'\times\Bcal} =(Sv,u)_{\Bcal'\times\Bcal}
\quad\text{for all }u,v\in\Bcal.
\end{equation}
As noted before, the existence of a frame for \(\Bcal\) already
implies that $\Bcal$ has to be reflexive. In particular, 
this means that $S$ is self-adjoint.

The frame operator $S$ is also 
uniformly elliptic and continuous, i.e.,
\begin{equation}\label{eq:FrameOpPos}
A\|u\|^2_{\Bcal}\leq(Su,u)_{\Bcal'\times\Bcal}\leq B\|u\|^2_{\Bcal}.
\end{equation}
By the closed range theorem \cite{gohbgol1}, the frame operator $S$ therefore
exhibits a bounded inverse $S^{-1}\colon\Bcal' \rightarrow \Bcal$ satisfying
\begin{equation}\label{eq:invFrameOpPos}
\frac 1 B\|u'\|^2_{\Bcal'}\leq(u',S^{-1}u')_{\Bcal'\times\Bcal}
\leq \frac 1 A\|u'\|^2_{\Bcal'}.
\end{equation}
Herein, the dual norm on $\Bcal'$ is defined 
as usual by 
\[
\|f\|_{\Bcal'}=\sup_{0\neq v\in\Bcal}
\frac{|(f,v)_{\Bcal'\times\Bcal}|}{\|v\|_{\Bcal}}.
\]

We define the \emph{canonical dual frame} by 
\begin{equation}\label{eq:canonicalDual}
\tilde{f}_i \isdef  S^{-1} f_i,\quad i\in I,
\end{equation}
which is a sequence of elements in $\Bcal$. 
Analogously to the arguments in \cite{xxlhar18}, we see that the 
canonical dual frame  $\{ \tilde{f}_i\}_{i\in I}$,  has the frame bounds 
$1/B$ and $1/A$, and allows reconstruction. 
Denoting the corresponding analysis operator by $\widetilde T^\star$ and 
the frame operator by $\widetilde S$, we have that $\widetilde T^\star = T^\star 
S^{-1}$
and, consequently, $\widetilde S = S^{-1}$. 

There holds the following continuous analogous property of the 
un\-der-de\-ter\-mined 
least squares problem as a straightforward generalization of the Hilbert space
frame case, i.e. \cite[Lemma VIII]{duffschaef1}, 
adding details to \cite[Theorem 4.3]{xxlhar18}.

\begin{lem.}\label{lem:2} Let \(\{f_i\}_{i\in I}\subset\Bcal'\) be a frame. 
  Then, the minimum norm 
pseudo-inverse $T^\dagger \colon \Bcal'\to\ell^2(I) $ is given by 
$T^\dag=T^\star S^{-1}= \widetilde{T}^\star$
with \(\|T^\dagger u'\|_{\ell^2}\leq A^{-1}\|u'\|_{\Bcal'}\) for any 
\(u'\in\Bcal'\). Particularly, there holds for every \(u'\in\Bcal'\) that 
\begin{equation}\label{eq:minNorm}
\|{\bs\beta}\|_{\ell^2}^2=\|T^\dagger u'\|_{\ell^2}^2
+\|{\bs\beta}-T^\dagger u'\|_{\ell^2}^2
\end{equation}
for any sequence \({\bs\beta}\in\ell^2(I)\) with \(u'=T{\bs\beta}\).
\end{lem.}

\begin{proof}
We have that
\[
T T^\dagger T = T T^\star S^{-1} T = T\] 
and 
\[
T^\dagger T T^\dagger
= T^\star S^{-1} T T^\star S^{-1} = T^\star S^{-1} = T^\dagger.
\] 
The claim on the bound of \(T^\dagger\) directly follows from the properties
of the canonical dual frame. 

The second claim is obtained according to
\[
\|{\bs\beta}\|_{\ell^2}^2 = 
\|T^\dagger u'+{\bs\beta}-T^\dagger u'\|_{\ell^2}^2
=\|T^\dagger u'\|_{\ell^2}^2
+2(T^\dagger u',{\bs\beta}-T^\dagger u')_{\ell^2}
+\|{\bs\beta}-T^\dagger u'\|_{\ell^2}^2
\]
by noticing that
\begin{align*}
(T^\dagger u',{\bs\beta}-T^\dagger u')_{\ell^2}&=
(T^\dagger u',{\bs\beta})_{\ell^2}-(T^\dagger u',T^\dagger u')_{\ell^2}
\\
&=(T{\bs\beta},S^{-1}u')_{\Bcal'\times\Bcal}
-(TT^\dagger u',S^{-1}u')_{\Bcal'\times\Bcal}\\
&=(u',S^{-1}u')_{\Bcal'\times\Bcal}-
(u',S^{-1}u')_{\Bcal'\times\Bcal}=0.
\end{align*}
\end{proof}

Equation~\eqref{eq:minNorm} means that the pseudo-inverse $T^\dagger$ fulfills a 
minimum norm property, as in the Hilbert space setting, cp.\ \cite{ole1},
see also \cite{meles77}.

In frame theory it is well known, see, e.g., \cite{stoev09}, that Banach spaces
with $\ell^2$-frames are equivalent to Hilbert spaces. Indeed, the properties 
\eqref{eq:FrameOpSym} and \eqref{eq:FrameOpPos} imply that \(\Bcal\) becomes a 
Hilbert space when endowed with the inner product
\begin{equation}\label{eq:InnerPB}
\langle u,v\rangle_{\Bcal}\isdef(Su,v)_{\Bcal'\times\Bcal}
=(T^\star u,T^\star v)_{\ell^2}.
\end{equation}
Similarly, we may introduce an inner product on \(\Bcal'\)
according to
\begin{equation}\label{eq:InnerPBprime}
\langle u',v'\rangle_{\Bcal'}\isdef(u',S^{-1}v')_{\Bcal'\times\Bcal}
=(T^\dagger u',T^\dagger v')_{\ell^2}.
\end{equation}
In particular, Equations \eqref{eq:InnerPB} and \eqref{eq:InnerPBprime}
imply that the Hilbert space structure on \(\Bcal\) is inherited from
\(\ell^2(I)\) using the analysis operator, while the one for \(\Bcal'\)
is inherited by the analysis operator of the canonical dual frame.
We denote the corresponding norms by
\[
  \vertiii{u}_{\Bcal}\isdef\sqrt{\langle u,u\rangle_\Bcal}
  \quad\text{and}\quad
  \vertiii{u'}_{\Bcal'}\isdef\sqrt{\langle u',u'\rangle_{\Bcal'}},
\]
respectively.
In particular, we have that any frame is a Parseval 
frame with respect to the Hilbert space topology, 
since by \eqref{eq:InnerPB} we have that 
\be \label{eq:starnew}
  \vertiii{u}_{\Bcal}^2 = \|T^\star u\|_{\ell^2}^2. 
\ee
The construction amounts to a generalization
of the canonical tight frame, which is done in Hilbert spaces by applying 
the square-root of the inverse frame operator on the primal frame, i.e., 
considering \(f_i^{t} \isdef S^{-1/2} f_i\), \(i\in I\). 
This definition, however, requires the uniformly elliptic, symmetric 
frame operator to%be an endomorphism and, hence, an isomorphism, which 
map a space into itself and therefore allowing the concatenation of 
$S^{-1/2}$ with itself. 
In the Banach space setting, we can consider ``the closest possible thing'', 
which is changing the topology in the aforementioned way. Clearly, we have 
$\langle S^{-1/2} f_i, S^{-1/2} f_j \rangle
= \langle f_i, S^{-1} f_j \rangle$ in the Hilbert space context.
As we have shown, using the inner products defined in 
%\eqref{eq:InnerPB} and 
\eqref{eq:InnerPBprime}, 
a frame \(\{f_i\}_i\subset\Bcal'\) 
becomes a Parseval frame for $\Bcal'$, while its canonical dual becomes 
a Parseval frame with respect to \eqref{eq:InnerPB}. 
Hence, let us note again, that we use a dual approach to finding the ``canonical
Parseval frames'', 
cp.\ \cite{ole1}, by not changing the frame elements, but rather adapting the
topology. 
Thereby, $\Bcal$ is isomorphic to a Hilbert space with the equivalent norm 
\(\vertiii{\cdot}_{\Bcal}\). 
There obviously holds
\[
\sqrt{A}\|u\|_\Bcal\leq\vertiii{u}_{\Bcal}\leq\sqrt{B}\|u\|_\Bcal
\quad\text{for all }u\in\Bcal.
\]
This stresses the same argument from a different viewpoint, 
when the frame bounds matter one should distinguish this definition of frames, 
and bases, from 
the Hilbert space concept. As a particular case for this argument, let us note 
that in the Hilbert space context, not all
frames are automatically considered to be Parseval frames.

Finally, let us mention a clear difference to Hilbert space frames. For the 
latter, the combination of the analysis and the synthesis operator in reversed 
order results in the (bi-infinite) Gram matrix. In the present setting, 
this is not possible, as $T^\star$ and $T$ cannot be combined in this order. 
Even so, the operator 
\[
{\bs G}\colon\ell^2(I)\to\ell^2(I),\quad 
{\bs G}\isdef T^\star \widetilde{T}=T^\star S^{-1}T=T^\dag T
\]
is a well defined, bounded
endomorphism of \(\ell^2(I)\). The canonical representation 
is by the (potentially bi-infinite) matrix 
\[
{\bs G}\isdef\big[\big(
f_i , \tilde f_j\big)_{\Bcal'\times\Bcal}\big]_{i,j\in I}.
\]
Particularly, there holds that \(\range{\bs G}
=\range{T^\star}\) and
\({\bs G}\colon\ell^2(I)\to\range{T^\star}\)
is an orthogonal projection, which follows from 
\[
  ({\bs G}{\bs\alpha},{\bs\alpha})_{\ell^2}
  =\big(T^\dag T{\bs\alpha},{\bs\alpha}\big)_{\ell^2}
  =\big(T{\bs\alpha},S^{-1}T{\bs\alpha}\big)_{\Bcal'\times\Bcal}
  =  \big(T^\dag T{\bs\alpha},T^\dag T{\bs\alpha}\big)_{\ell^2}
=  ({\bs G}{\bs\alpha},{\bs G}{\bs\alpha})_{\ell^2}
  ,
\]
according to \eqref{eq:InnerPBprime}.

\subsection{Banach Riesz bases with square-summable coefficients}
Introducing a notion of independence, we can use this to give a more precise
characterization of the Hilbert space topology introduced in the previous 
paragraph. This will be detailed in the following.

\begin{def.}
  A sequence \({\bs f}\isdef\{f_i\}_{i\in I}\subset\Bcal'\) is a 
  \emph{Riesz sequence} in $\Bcal'$, if there exist \(A,B>0\) such
  that for all finite coefficient sequences 
  \({\bs\alpha}=\{\alpha_i\}_{i\in I}\) there
holds
\begin{equation}\label{eq:frame}
A\|{\bs\alpha}\|_{\ell^2}^2\leq
\bigg\|\sum_{i\in I} \alpha_if_i\bigg\|_{\Bcal'}^2
\leq B\|{\bs\alpha}\|_{\ell^2}^2.
\end{equation}
It is called a \emph{Riesz basis} if it is a Riesz sequence 
and \(\Bcal'=\overline{\finspan{f_i}_{i\in I}}\) .
\end{def.}

\begin{remark} The existence of an \(\ell^2\)-frame for \(\Bcal\), see
Definition~\ref{def:l2frame}, is equivalent to the existence of a 
Riesz basis for \(\Bcal'\). This is due to the fact that a Banach space
having an \(\ell^2\)-frame renders it to be isomorphic to a separable Hilbert
space, which exhibits an orthonormal basis. The latter is particularly a Riesz 
basis. Vice versa, each Riesz basis constitutes a frame, see also below. 
\end{remark}

The following result from \cite{chst03} characterizes
the interplay between a frame for \(\Bcal\) and a
Riesz basis for \(\Bcal'\).

\begin{theorem}\label{thm:charRiesz} Let \({\bs f}\subset\Bcal'\) 
  be a frame.
The following statements are equivalent:
\begin{itemize}
  \item[(i)] The sequence \({\bs f}\subset\Bcal'\) is a Riesz basis. 
  %{\sf (basis)}
\item[(ii)] The synthesis operator is injective, i.e.,
  if \(T{\bs\alpha}=0\), then \({\bs\alpha}={\bs 0}\).
  %{\sf ($\omega$-independence)}
\item[(iii)] The analysis operator is surjective, i.e.,
  $\range{T^\star}=\ell^2(I)$. %{\sf (onto analysis)}
\item[(iv)] There exists a
  biorthogonal
  sequence \(\widetilde{\bs f}\isdef\{\tilde{f}_i\}_{i\in I}\subset\Bcal\), 
  i.e., \((f_i,\tilde{f}_j)_{\Bcal'\times\Bcal}=\delta_{i,j}\). 
The biorthogonal sequence is 
unique and forms a
Riesz basis in \(\Bcal\). %{\sf (biorthogonal Riesz)}
\end{itemize}
\end{theorem}

One can easily show that the dual Riesz basis $\widetilde{\bs f}$ is 
given by
\eqref{eq:canonicalDual}. Indeed, there holds
\begin{equation}\label{eq:dualRieszCharact}
  S \tilde{f}_j = \sum_{i\in I}
  (f_i,\tilde{f}_j)_{\Bcal'\times\Bcal}f_i = f_j.
\end{equation}
Hence, for Riesz bases,
we have the following characterization of the inner products
\eqref{eq:InnerPB} and \eqref{eq:InnerPBprime}, respectively.
\begin{lem.}
  Let \({\bs f}\subset\Bcal'\) be a Riesz basis
  with biorthogonal basis
  \(\widetilde{\bs f}\). Then, there holds 
\[
\langle f_i,f_j\rangle_{\Bcal'}
=\big\langle\tilde{f}_i,\tilde{f}_j\big\rangle_{\Bcal}=\delta_{i,j}.
\]
Hence, these bases amount to orthonormal bases of \(\Bcal'\) and \(\Bcal\), 
respectively. 
\end{lem.}

\begin{proof}
From Theorem~\ref{thm:charRiesz}, we have the existence and uniqueness of the
biorthogonal basis. From \eqref{eq:dualRieszCharact}, 
we have \(S\tilde{f}_i=f_i\)
or, equivalently, \(\tilde{f}_i=S^{-1}f_i\).
Thus, we obtain 
\[
\langle f_i,f_j\rangle_{\Bcal'}=(f_i,S^{-1}f_j)_{\Bcal'\times\Bcal}
=(f_i,\tilde{f}_j)_{\Bcal'\times\Bcal}=\delta_{i,j}
\]
and, analogously,
\[
\langle\tilde{f}_i,\tilde{f}_j\rangle_{\Bcal}
=(SS^{-1}f_i,S^{-1}f_j)_{\Bcal'\times\Bcal}
=(f_i,\tilde{f}_j)_{\Bcal'\times\Bcal}=\delta_{i,j}.
\]
\end{proof}
The converse is also true, if ${\bs f}\subset \Bcal'$ is an orthonormal 
basis with the Hilbert space topology, it is a Riesz basis within the 
Banach space topology. 

\subsection{Matrices acting on sequences in Banach spaces}
In this section, let \(K\subset\Nbb\) also denote a finite or infinite 
index set.
Given a matrix ${\bs M}=[m_{k,i}]_{k\in K, i\in I}\in \mathbb{R}^{K \times I}$, 
such that the rows form 
$\ell^2$-sequences, and a Bessel sequence for $\Bcal$ denoted as 
${\bs\Phi}=\{ \varphi_i \}_{i \in I}$, we define the action of ${\bs M}$ on 
$\bs\Phi$, denoted by ${\bs M}{\bs\Phi}$, through
\begin{equation} \label{eq:matrseq1} 
[{\bs M}{\bs\Phi}]_k 
\isdef \sum \limits_{i \in I} m_{k,i}\varphi_i\subset\Bcal'
\quad\text{for all }k\in K.
\end{equation}
Intuitively, we consider sequences of functions as column vectors and
define the action of a matrix by the usual matrix-vector product.
With this, we get a mapping ${\bs M}$ from the sets of Bessel sequence for 
$\Bcal$ into the set of sequences in $\Bcal'$. By assumption, 
\eqref{eq:matrseq1} is well-defined, also for infinite index sets.

Using this notation we can reinterpret the reconstruction property of 
frames as follows.
\begin{lem.} Let ${\bs f}\subset\Bcal'$ be a frame
  for a Banach space \(\Bcal\). Then, there holds that 
  $$ f_i = [{\bs G}{\bs f}]_i.
  $$
\end{lem.}
\begin{proof}
We have
  \[
    [{\bs G}{\bs f}]_i=\sum_{j\in I}g_{i,j}f_j=
    \sum_{j\in I}g_{j,i}f_j=
        \sum_{j\in I}(f_j,\tilde{f}_i)_{\Bcal'\times\Bcal}f_j=S\tilde{f}_i=f_i,
  \]
  due to the symmetry of \({\bs G}\), cp.\ \eqref{eq:InnerPBprime},
  the definition of the frame operator \eqref{eq:frameOp} and the
  definition of the canonical dual frame \eqref{eq:canonicalDual}.
 \end{proof}

It is easy to see that if ${\bs M}\in\Rbb^{I\times I}$ is a bounded operator 
on $\ell^2(I)$, cp.\  \cite{cron71},
and thereby fulfilling the criteria above, then, given
a Bessel sequence ${\bs\Phi}=\{ \varphi_i \}_{i \in I}$, the sequence
${\bs M}{\bs\Phi}$ is again a Bessel sequence with bound 
$B\norm{\ell^2 \rightarrow \ell^2}{\bs M}$. 
In this case the global coefficient of the new sequence is the matrix 
multiplication of the local coefficients as 
\begin{equation}\label{eq:matrcoeff1} 
\begin{aligned} 
([{\bs M}{\bs\Phi}]_k,v)_{\Bcal'\times\Bcal} &=
\left(\sum_{i\in I}m_{k,i}\varphi_i,v\right)_{\Bcal'\times\Bcal}
=\sum_{i\in I}m_{k,i}(\varphi_i,v)_{\Bcal'\times\Bcal} \\
&=  \langle {\bs m}_k, T^\star_{\bs\Phi} v \rangle_{\ell^2}=
[{\bs M}T^\star_{\bs\Phi} v]_k,
\end{aligned}
\end{equation}
where \(T^\star_{\bs\Phi}\colon\Bcal'\to\ell^2(I)\) is the analysis operator
associated to \({\bs\Phi}\)and ${\bs m}_k$ the $k$-th row of the matrix 
${\bs M}$, i.e. \([{\bs m}_k]_i = m_{k,i} \). 
This means that the analysis operator of \({\bs M\bs\Phi}\) 
is given by \(T^\star_{{\bs M\bs\Phi}}={\bs M}T^\star_{\bs\Phi}\).
Similarly, we obtain for the corresponding
frame operator that 
\begin{equation}\label{eq:transformedFrameOp}
S_{\bs M\bs \Phi} = T_{\bs\Phi}{\bs M}^\star{\bs M}T^\star_{\bs\Phi}.
\end{equation}

Let ${\bs f}\subset\Bcal'$ be a frame for \(\Bcal\) with bounds $A,B$ and let 
${\bs M}\colon\ell^2(I)\to\ell^2(I)$ have closed range.
Then there exits bounds $0<\underline{c}\leq\overline{c}<\infty$ such that
\begin{equation} \label{eq:star} \underline{c} \norm{\ell^2}{{\bs\alpha}}^2 
  \le \norm{\ell^2}{\bs M \bs\alpha}^2 \le 
\overline{c}\norm{\ell^2}{\bs\alpha}^2\quad\text{for all }
{\bs\alpha} \in \kernel{\bs M}^\perp.
\end{equation}
Hence, if $\range{T^\star} \subseteq \kernel{\bs M}^\perp$ for a frame $\bs f$
with bounds $A,B$, we have that
${\bs M}{\bs f}$ is a frame with bounds $\underline{c} A, \overline{c} B$. 

The next lemma gives a characterization of the resulting frame in case that the
underlying matrix is an isometry.
\begin{lem.}\label{lem:basisTrafo}
Let \({\bs f}\subset\Bcal'\) be a frame for \(\Bcal\) and let 
\({\bs U}\colon\ell^2(I)\to\ell^2(I)\) be an isometry, i.e.,
\({\bs U}^\star{\bs U}=\identity{\range{T^\star}}\).
Then 
\[
\psi_i\isdef\sum_{j\in I}u_{i,j}f_{j},\quad i\in I,
\]
is a Parseval frame with respect to the topology induced by the
inner product \(\langle\cdot,\cdot\rangle_{\Bcal'}\).
Analogously, 
\[
  \widetilde{\psi}_i\isdef\sum_{j\in I}u_{i,j}S^{-1}f_{j}=
\sum_{j\in I}u_{i,j}\tilde{f}_{j},\quad i\in I,
\]
is a Parseval frame with respect to the topology induced by the
inner product \(\langle\cdot,\cdot\rangle_{\Bcal}\).
In particular, the sequences \(\{\psi_i\}_{i\in I}\) 
and \(\{\tilde{\psi}_i\}_{i\in I}\) are dual frames in the
$\Bcal'\times \Bcal$-sense.

If, additionally, \({\bs f}\subset\Bcal'\) is a Riesz basis for \(\Bcal'\)
and \({\bs U}\colon\ell^2(I)\to\ell^2(I)\) is unitary, i.e.,
\({\bs U}{\bs U}^\star={\bs U}^\star{\bs U}=\identity{\ell^2(I)}\),
then
\[
  \langle \psi_i,\psi_j\rangle_{\Bcal'}=\delta_{i,j}\quad\text{and}\quad
\big\langle \widetilde{\psi}_i,\widetilde{\psi}_j\big\rangle_{\Bcal}
=\delta_{i,j}\quad
\text{for }i,j\in I,
\]
respectively. This means that the sequences \(\{\psi_i\}_{i\in I}\) 
and \(\{\tilde{\psi}_i\}_{i\in I}\) are orthonormal bases for \(\Bcal'\)
and \(\Bcal\), respectively.
\end{lem.}

\begin{proof}
The sequence ${\bs f}$ forms a Parseval frame for the Hilbert space $\Bcal'$ 
with respect to Hilbert space topology, cp.\ \eqref{eq:starnew}. 
We set $\bs\Psi \isdef {\bs U}{\bs f} $ and 
$\widetilde{\bs\Psi}\isdef {\bs U}\widetilde{\bs f}$, respectively. In view
of \eqref{eq:transformedFrameOp}, this implies
\(S_{\bs\Psi}=S_{\bs f}\) and \(S_{\widetilde{\bs\Psi}}=S_{\widetilde{\bs f}}\),
respectively. Hence, they are
Parseval frames for $\Bcal'$ and $\Bcal$, respectively, with respect to the
Hilbert space topology and dual to each other in the $\Bcal' \times \Bcal$-sense. 

Now let $\bs f$ be a Riesz basis, then, 
with the convention 
\be \label{eq:starstarnew} \langle{\bs f},{\bs f}^\intercal\rangle_{\Bcal}\isdef[
\langle f_i,f_j\rangle_{\Bcal}]_{i,j\in I},
\ee 
we obtain
\[
\langle{\bs\Psi},{\bs\Psi}^\intercal\rangle_{\Bcal}
={\bs U}\langle{\bs f},{\bs f}^\intercal\rangle_{\Bcal}{\bs U^\star}
={\bs U}\identity{\ell^2(I)}{\bs U}^\star
={\bs U}{\bs U}^\star=\identity{\ell^2(I)}.
\]
The proof for \(\widetilde{\bs\Psi}\) is analogous.
\end{proof}

For Riesz bases we can also go beyond the isometric setting, by using continuous,
invertible operators \({\bs M}\colon\ell^2(I)\to\ell^2(I)\).
\begin{cor.}
\label{cor:basisTrafo}
Let \({\bs f}\subset\Bcal'\) be a Riesz basis and let 
\({\bs M}\colon\ell^2(I)\to\ell^2(I)\) be continuous and invertible. 
Then $\widetilde{{\bs M}{\bs f}} = {\bs M}^{-1} \widetilde{\bs f}$. 
\end{cor.}

%%%%%%%%%%%%%%%%%%%%%%%%%%%%%%%%%%%%%%%%%%%%%%%%%%%%%%%%%%%%%%%%%%%%%%%%%%%%%%%%
\section{Finite sets of functionals}\label{sec:GenFinSamp}
%%%%%%%%%%%%%%%%%%%%%%%%%%%%%%%%%%%%%%%%%%%%%%%%%%%%%%%%%%%%%%%%%%%%%%%%%%%%%%%%

In \cite{HM22}, a finite set of points 
$\{{\bs x}_1, \ldots,{\bs x}_N\}\subset\Rbb^d$ was considered,
and the space spanned by the Dirac-\(\delta\)-distributions 
$\delta_{{\bs x}_i}$, \(i=1,\ldots,N\) was used to define
samplets. Our goal is to extend this approach to more 
general finite sets of functionals in \(\Bcal'\). 

To this end, we start from a finite set of functionals
\({\bs f}_N\isdef \{f_{1},\ldots,f_{N}\}\subset\Bcal'\). 
We denote the linear subspace
spanned by these functionals by \(\Xcal_N'\isdef\finspan{{\bs f}_{N}}\).

Certainly, we can investigate the \emph{finite analysis operator}
\[
T_{N}^\star\colon\Bcal\to\Rbb^N,\quad 
T_{{N}}^\star v\isdef\big\{(f_{i},v)_{\Bcal'\times\Bcal}\big\}_{i=1}^N,
\]
as well as the \emph{finite synthesis operator}\ 
\[
T_{{N}}\colon\Rbb^N \to\Xcal_N'\subset\Bcal',\quad
T_{{N}}{\bs\alpha}\isdef\sum_{i=1}^N\alpha_if_{i}.
\]
Their concatenation then gives rise to the \emph{finite frame operator}
\begin{equation}\label{eq:restrictedS}
S_{{N}}\isdef T_{{N}} T_{{N}}^\star
\colon\Bcal\to\Xcal_N'\subset\Bcal',\quad S_{{N}} v
=\sum_{i=1}^N(f_{i},v)_{\Bcal\times\Bcal'}f_{i}.
\end{equation}

\begin{remark}
We remark that the setting considered here is akin to the one
considered in \emph{optimal recovery}. Optimal recovery deals
with the question of optimally approximating maps of functions,
belonging to a specific class, from limited information. In this context,
the functionals \(f_{i_1},\ldots,f_{i_N}\) 
are then referred to as \emph{information functionals}, while the analysis
operator \(T_N^\star\) is referred to as 
\emph{information operator}.
In particular, the synthesis operator \(\widetilde{T}_N=(T_N^\star)^\dagger\) 
of the canonical dual frame constitutes 
an optimal algorithm for the recovery of an object \(v\in\Bcal\) from
the information \(T_N^\star v\), cp.\ \cite{OS19}.
For further details, we refer to the survey
 \cite{MR77}.
\end{remark}

Since the finite frame operator \eqref{eq:restrictedS} has finite-dimensional 
range, it is particularly continuous and its range is closed.
As \(\Xcal_N'\) is finite dimensional, we can
select a subset ${\bs f}_M\subset{\bs f}_N$
with \(M\leq N\), which is a basis for $\Xcal_N'$, i.e., by the reflexivity we 
have for all
$u' \in \Xcal'_N$ 
\[
u' 
= \sum \limits_{i = 1}^M 
(u', \hat f_i)_{\Bcal'\times \Bcal} f_i,
\]
where $\hat{f}_i$ denote the coefficient functionals on $\Xcal_N'$. By the 
Hahn-Banach theorem, they can be extended to functionals on $\Bcal'$.
Considering this sequence $\widehat{{\bs f}_M} = \{ \hat{f}_i \} \subseteq\Bcal$
and its finite frame operator $\widehat S_M$ we have that 
\[
\Bcal'=\ker(\widehat{S}_M)
\oplus{\range{S_N}}\quad\text{and}\quad
  \Bcal=\ker(S_N) \oplus{\range{\widehat{S}_M}}.
\]
We set 
\[\Xcal_N\isdef\range{\widehat S_M}\]

and note that the restriction
\(
S_N\colon \Xcal_N\to\Xcal_N'
\)
is an invertible operator. 
Especially, denoting its smallest eigenvalue by \(A_N\)
and its largest eigenvalue by \(B_N\), respectively, it is evident that 
\({\bs f}_N\) is a frame for \(\Xcal_N\subset\Bcal\)
with frame bounds \(A_N, B_N\).
As a consequence, Lemma~\ref{lem:basisTrafo} applies with the obvious modifications 
also to the present finite dimensional situation.
Even so, the stability is not necessarily guaranteed for \(N\to\infty\).

However, under the stronger condition that 
\({\bs f}_N\subset{\bs f}\subset\Bcal'\),
where \({\bs f}\) is a Riesz basis for \(\Bcal\), we have that 
\({\bs f}_N\) is a Riesz sequence itself, see \cite{ole1}.  
Moreover, the inner products \(\langle\cdot,\cdot\rangle_{\Bcal}\) on 
\(\Xcal_N\) and \(\langle\cdot,\cdot\rangle_{\Bcal'}\) on \(\Xcal_N'\),
respectively, are fully characterized by the sequence \({\bs f}_N\). 
In this case, the sequence introduced above satisfies
$\widehat{{\bs f}_M} = \widetilde{{\bs f}_N}$ with $M=N$, i.e., 
$\Xcal_N = \finspan{\widetilde{{\bs f}_N}}$.

\begin{lem.} \label{lem:4} 
Let ${\bs f}$ be a Riesz basis for $\Bcal'$ and let 
\({\bs f}_N =\{f_{i_1},\ldots,f_{i_N}\}\subset{\bs f}\) be a finite subset 
with $\Xcal_N'=\finspan{{\bs f}_N}$. 
Then ${\bs f}_N$ is a Riesz basis for $\Xcal_N'$ with biorthogonal basis
\(\widetilde{\bs f}_N\subset\Xcal_N\), where 
\(\tilde{f}_{i_j}=S^{-1}{f}_{i_j} = S_N^{-1}{f}_{i_j}\). 
Furthermore, letting
\(\langle u,v\rangle_{\Xcal_N'} 
\isdef\big(u,S_{N}^{-1}v\big)_{\Bcal'\times\Bcal}\),
there holds for any \(u',v'\in\Xcal_N'\) that
\[
\langle u',v'\rangle_{\Xcal_N'}=\langle u',v'\rangle_{\Bcal'}
\quad\text{for all
}u',v'\in\Xcal_N'.
\]
The analogous result holds for \(\langle u,v\rangle_{\Xcal_N} 
\isdef\big(S_Nu,v\big)_{\Bcal'\times\Bcal}\).
\end{lem.}
\begin{proof}Since \(\Xcal_N'\) is finite dimensional, it suffices to show
that the basis ${\bs f}_N$ %\(f_{i_1},\ldots,f_{i_N}\in\Xcal\)
satisfies 
\(\langle f_{i_j},f_{i_k}\rangle_{\Xcal_N'}=\delta_{j,k}\). 
Let \(\tilde{f}_{i_j}=S^{-1}f_{i_j}\) be the corresponding subsequence of 
the canonical dual basis. There holds
\[
S_{N}\tilde{f}_{i_j}=
\sum_{k=1}^N(f_{i_k},\tilde{f}_{i_j})_{\Bcal'\times\Bcal}f_{i_k}=f_{i_j}
\]
due to the biorthogonality from Theorem~\ref{thm:charRiesz}. Hence, we
have 
\[
S_{N}^{-1}f_{i_j}=\tilde{f}_{i_j}=S^{-1}f_{i_j}\quad\text{for }j=1,\ldots, N,
\] and, 
therefore,
\[
\langle f_{i_j},f_{i_k}\rangle_{\Xcal_N'}
=(f_{i_j},S_N^{-1}f_{i_k})_{\Bcal'\times\Bcal}
=(f_{i_j},S^{-1}f_{i_k})_{\Bcal'\times\Bcal}
=(f_{i_j},\tilde{f}_{i_k})_{\Bcal'\times\Bcal}=\delta_{j,k},
\]
as claimed. 
\end{proof}

We remark that, in the situation of the previous lemma, the space 
\(\big(\Xcal_N',\langle\cdot,\cdot\rangle_{\Xcal_N'}\big)\) is always 
isometrically isomorphic to \(\Rbb^N\) and we have
\[
\langle u',v'\rangle_{\Xcal_N'}
=\left\langle\sum_{j=1}^N\alpha_j
f_{i_j},\sum_{j=1}^N\beta_jf_{i_j}\right\rangle_{\Xcal_N'}
={\bs\alpha}^\intercal{\bs\beta}.
\]

%%%%%%%%%%%%%%%%%%%%%%%%%%%%%%%%%%%%%%%%%%%%%%%%%%%%%%%%%%%%%%%%%%%%%%%%%%%%%%%%
\section{Construction of generalized samplet bases}\label{sct:construction}
%%%%%%%%%%%%%%%%%%%%%%%%%%%%%%%%%%%%%%%%%%%%%%%%%%%%%%%%%%%%%%%%%%%%%%%%%%%%%%%%
\subsection{Multiresolution analysis of functionals}
From now on, let \(\Omega\subset\Rbb^d\) be a domain.
For the construction of samplet bases, we make the assumption that
\(\Bcal'\) is imbedded into the space of \emph{compactly supported
distributions}, i.e., \(\Bcal'\subset\Ecal'(\Omega)\), where \(\Ecal(\Omega)
\isdef C^\infty(\Omega)\) is endowed with the usual limit of Fr\'echet spaces 
topology. This assumption is, for example, satisfied if \(\Ecal(\Omega)\) is a
dense subspace of \(\Bcal\). For a given distribution \(f\in\Ecal'(\Omega)\),
we define its \emph{restriction} \(f_O\) to \(O\subset\Omega\) by the property
\[
  (f_O,v)_{\Bcal'\times\Bcal}=  (f,v)_{\Bcal'\times\Bcal}
  \quad\text{for all }v\in\Ecal(O),
\]
with the obvious convention that \(\Ecal(O)\) is considered as a subspace of
\(\Ecal(\Omega)\). Then, its \emph{support} is defined as
\[
  \support{f}\isdef\{{\bs x}\in\Omega: 
\text{there is no open set }O\ni{\bs x}\text{ such that }
f_O=0\}.
\]
For example, we have \(\support{\delta_{\bs x}}=\{{\bs x}\}\).
For the particular case \(f\in\Ecal(\Omega)\), the support of the
distribution
\[
  (T_f,v)_{\Bcal'\times\Bcal}=(T_f,v)_\Omega\isdef\int_\Omega fv\d{\bs x}
\]
coincides with the usual support of the function \(f\).

Now let \({\bs f}_N=\{f_1,\ldots, f_N\}\subset\Bcal'\) be a finite set of 
functionals. As before we set 
\[
\Xcal'\isdef\Xcal_N'=\finspan{{\bs f}_N}.
\]
The construction of samplet bases relies on a multiresolution
analysis
\[
\Xcal_0'\subset\Xcal_1'\subset\cdots\subset\Xcal_J'=\Xcal', 
\]
constructed by clustering the functionals with respect to the similarity of 
their supports, 
obtained from a suitable \emph{cluster tree}, as defined below. 

\begin{def.}\label{def:cluster-tree}
Let \(X\) be a set and let $\mathcal{T}=(V,E)$ be a tree with vertices $V$ 
and edges $E$.
We define its set of leaves as
\(
\mathcal{L}(\mathcal{T})\isdef\{\nu\in V\colon\nu~\text{has no children}\}.
\)
The tree $\mathcal{T}$ is a \emph{cluster tree} for
\(X\), if
\(X\) is the {root} of $\mathcal{T}$ and
all $\nu\in V\setminus\mathcal{L}(\mathcal{T})$
are disjoint unions of their children.
The \emph{level} \(j_\nu\) of $\nu\in\mathcal{T}$ is its distance from
the root. The \emph{depth} of \(\Tcal\) is given by 
\(J\isdef\max\{j_\nu: \nu\in\Tcal\}\). 
Finally, we define the set of clusters
on level $j$ as
\(
\mathcal{T}_j\isdef\{\nu\in\mathcal{T}\colon \nu~\text{has level}~j\}.
\)
\end{def.}

In the original work \cite{HM22}, exclusively the situation  
\({\bs f}_N=\{\delta_{{\bs x}_1},\ldots,\delta_{{\bs x}_N}\}\) has been
considered, cp.\ Example~\ref{ex:RKHS}.
The cluster tree was then constructed by longest edge bisection of
the bounding box of the 
supports of the Dirac-$\delta$-distributions with respect to the
Wasserstein distance
\(d_W(\delta_{{\bs x}_i},\delta_{{\bs x}_j})=\|{\bs x}_i-{\bs x}_j\|_2\).
In view of more general functionals, we resort here
to a spectral clustering approach, see \cite{vLux07} and the
references therein. For a more general view on the topic of clustering
in abstract spaces, we refer to \cite{TSS15}, see also \cite{KW78}.
For the reader's convenience, we recall the ideas from \cite{vLux07}.
To this end, we assume that \(d\colon\Xcal'\times\Xcal'\to[0,\infty)\)
is a (pseudo-) metric, which reflects the distance between the supports
of the distributions under consideration, for example,
\[
  d(f_i,f_j)=\mydist{\support{f_i},\support{f_j}},
\]
where \(\mydist{A,B}\isdef\inf\{\|{\bs a}-{\bs b}\|:{\bs a}\in A,{\bs b}\in B\}\)
is the usual distance of sets.
We introduce the 
\emph{pairwise similarities}
\[
  w_{ij}\isdef s(f_i,f_j),
\]
for some symmetric similarity function \(s\colon\Xcal'\times\Xcal'\to[0,\infty)\).
Collecting the pairwise similarities yields the data similarity matrix
\[
{\bs W}\isdef[w_{i,j}]_{i,j=1}^N.
\]

We mention some common examples of similarity functions in the following:
\begin{enumerate}
    \item 
The choice
\[
s(f_i,f_j)\isdef\begin{cases}1,&\text{if }d(f_i,f_j)<\varepsilon,\\
0,&\text{else}\end{cases}
\]
makes \({\bs W}\) correspond to the adjacency matrix of the 
\emph{$\varepsilon$-neighborhood graph}. 
\item Setting
\[
  s(f_i,f_j)\isdef\begin{cases}1,&\text{if $f_i\in\textnormal{$k$NN}(f_j)$
     or $f_j\in\textnormal{$k$NN}(f_i)$},\\
0,&\text{else},
\end{cases}
\]
where $f_i\in\textnormal{$k$NN}(f_j)$ is the set of \(f_i\)'s $k$-nearest
neighbors, 
makes \({\bs W}\) the adjacency matrix of the 
\emph{mutual $k$-nearest neighbors graph}. 
\item Finally, one may consider globally
supported similarity functions, such as the
Gaussian kernel
\[
  s(f_i,f_j)\isdef e^{-\frac{d(f_i,f_j)^2}{2\ell^2}},\quad\text{for }
  \ell>0,
\]
resulting in fully connected, weighted similarity graphs.
\end{enumerate}
We introduce the diagonal
matrix \({\bs D}\isdef\mydiag{[\sum_{j=1}^N w_{ij}]_{i=1}^N}\)
and the \emph{unnormalized graph Laplacian}
\[
  {\bs L}\isdef{\bs D}-{\bs W}.
\]
There holds the following result, see, for example, \cite[Section 2]{Moh91}
or \cite[Propositions 1 and 2]{vLux07}.

\begin{lem.} The unnormalized graph Laplacian satisfies
\[
  {\bs x}^\intercal{\bs L}{\bs x}=\frac 1 2\sum_{i,j=1}^N w_{ij}(x_i-x_j)^2.
\]
Consequently, \({\bs L}\) is symmetric and positive semi-definite and 
the vector \({\bs 1}=[1,\ldots,1]^\intercal\in\Rbb^N\) is an eigenvector
corresponding to the eigenvalue $0$. The multiplicity of the eigenvalue $0$
corresponds to the number of connected components of the graph associated
to \({\bs W}\).
\end{lem.}

Subdividing the graph corresponding to \({\bs W}\)
by means of the \emph{Fiedler vector}, i.e., the eigenvector
corresponding to the second biggest eigenvalue of \({\bs L}\), yields 
a bisection of the set \(X\), cp.\ \cite{Fie73}. The procedure is
detailed in Algorithm~\ref{algo:bisection}, cf.\ \cite{vLBB08}.
The latter reference particularly provides conditions under which the
bisection becomes stable in the limit \(N\to\infty\).
\begin{figure}[htb]
\small
\begin{center}
\begin{minipage}{\textwidth}
\begin{algorithm}[H]
	\caption{Spectral bisection algorithm}
	\label{algo:bisection}	
  \KwData{Similarity matrix \({\bs W}\)}
  \KwResult{Clusters \(\tau_{1},\tau_{2}\)}
	
	\Begin{
    Compute the eigenvector \({\bs v}\) corresponding to the second
    eigenvalue of \({\bs L}\).

    set \(\tau_{1}\isdef\{i:v_i \geq 0\}\) and \(\tau_{2}\isdef\{i:v_i < 0\}\)
		}
\end{algorithm}
\end{minipage}
\end{center}
\end{figure}

Recursively applying Algorithm~\ref{algo:bisection} for each of the
resulting subgraphs obtained by deleting all the edges connecting 
\(\tau_{1}\) and  \(\tau_{2}\), with the correspondingly
modified adjacency matrices \({\bs W}_{\tau_{1}}\),\({\bs W}_{\tau_{2}}\)
and Laplacian 
\({\bs L_{\tau_{1}}}\), \({\bs L_{\tau_{2}}}\), 
yields the desired cluster tree,
cf.\ Definition \ref{def:cluster-tree}. 
%%%%%%%%%%%%%%%%%%%%%%%%%%%%%%%%%%%%%%%%%%%%%%%%%%%%%%%%%%%%%%%%%%%%%%%%%%%%%%%%
\subsection{Computation of the basis transformation}
%%%%%%%%%%%%%%%%%%%%%%%%%%%%%%%%%%%%%%%%%%%%%%%%%%%%%%%%%%%%%%%%%%%%%%%%%%%%%%%%
In Section~\ref{sec:GenFinSamp}, we have seen that the functionals 
\({\bs f}_N\) form a Parseval frame with respect to the topology induced by
\(S_N\) for an appropriately chosen space \(\Xcal_N\subset\Bcal\).
In view of Lemma~\ref{lem:basisTrafo}, a generalized samplet basis is
now obtained by constructing a suitable isometry \({\bs U}\in\Rbb^{N\times N}\).

The cluster tree provides a hierarchical structure that we
exploit in the construction of the samplet basis.
We introduce a \emph{two-scale} 
transform between functionals associated to a cluster $\nu$ of level
$j$ and its son clusters on level $j+1$. 
To this end, we create \emph{scaling functionals}
$\mathbf{\Phi}_{j}^{\nu} = \{ \varphi_{j,k}^{\nu} \}_k$ and
\emph{samplets} $\mathbf{\Psi}_{j}^{\nu} = \{\psi_{j,k}^{\nu} \}_k$
as linear combinations of the scaling 
functionals $\mathbf{\Phi}_{j+1}^{\nu}$ of $\nu$'s child clusters. 
Denoting the number of elements by \(n_{j+1}^\nu
\isdef|\mathbf{\Phi}_{j+1}^{\nu}|\), this results in the
\emph{refinement relations}
\[
\varphi_{j,k}^{\nu}
=\sum_{\ell=1}^{n_{j+1}^\nu}q_{j,\Phi,\ell,k}^{\nu}\varphi_{j+1,\ell}^{\nu}
\quad\text{and}\quad
\psi_{j,k}^{\nu}
=\sum_{\ell=1}^{n_{j+1}^\nu}q_{j,\Psi,\ell,k}^{\nu}
\varphi_{j+1,\ell}^{\nu},
\]
for certain coefficients \(q_{j,\Phi,\ell,k}^{\nu}
=[{\bs Q}_{j,\Phi}^{\nu}]_{l,k}\) and
\(q_{j,\Psi,\ell,k}^{\nu}=[{\bs Q}_{j,\Psi}^{\nu}]_{l,k} \).
As before, collecting functionals as column vectors, these relations
may be written in matrix notation as
\begin{equation}\label{eq:refinementRelation}
%%%%%%%%%%%%%%%%%%%%%%%%%%%%%%%%%%%%%%%%%%%%%%%%%%%%%%%%%%%%%%%%%%%%%%%%%%%%%%%%
\begin{bmatrix}\mathbf{\Phi}_{j}^{\nu}\\ \mathbf{\Psi}_{j}^{\nu}
\end{bmatrix}
 \isdef 
 [{\bs Q}_j^{\nu}]^\intercal\mathbf{\Phi}_{j+1}^{\nu}
=
% \big[ {\bs Q}_{j,\Phi}^{\nu},{\bs Q}_{j,\Psi}^{\nu}\big]^\intercal
 \begin{bmatrix} {\bs Q}_{j,\Phi}^{\nu} \\ {\bs Q}_{j,\Psi}^{\nu}
\end{bmatrix}
  \mathbf{\Phi}_{j+1}^{\nu}.
\end{equation}

In order to provide vanishing moments with respect to a finite-dimensional space
of primitives \(\Pcal\subset\Bcal\) and orthonormality in the context of 
Lemma~\ref{lem:basisTrafo}, the transformation \({\bs Q}_{j}^{\nu}\) has to be
appropriately constructed. For this purpose, we consider an orthogonal
decomposition of the \emph{moment matrix}
\begin{equation}\label{eq:momentMatrix}
  {\bs M}_{j+1}^{\nu}
  \isdef
  \big\langle {\bs p}_{m_q}, 
    {\left(\mathbf{\Phi}_{j+1}^{\nu}\right)}^\intercal
    \big\rangle_{\Bcal'\times\Bcal} \in\Rbb^{m_{\Pcal}\times n_{j+1}^\nu}
\end{equation}
with the notation adopted from \eqref{eq:starstarnew} and
\({\bs p}_{m_{\Pcal}}\isdef\{p_1,\ldots, p_{m_{\Pcal}}\}\), 
\(m_{\Pcal}\isdef\dim(\Pcal)\), being a basis of \(\Pcal\).
A canonical choice for \(\Pcal\)
is the set of polynomials up to degree up to degree $q$.
Following the idea
from \cite{AHK14}, we employ the QR decomposition, which results in
samplets with an increasing number of vanishing moments as outlined 
subsequently. Letting
\begin{equation}\label{eq:QR} 
%%%%%%%%%%%%%%%%%%%%%%%%%%%%%%%%%%%%%%%%%%%%%%%%%%%%%%%%%%%%%%%%%%%%%%%%%%%%%%%%
  ({\bs M}_{j+1}^{\nu})^\intercal  = {\bs Q}_j^\nu{\bs R}
  \defis\big[{\bs Q}_{j,\Phi}^{\nu} ,
  {\bs Q}_{j,\Psi}^{\nu}\big]{\bs R},
 \end{equation}
the moment matrix 
for the cluster's own scaling functionals and samplets is given by 
\begin{equation}\label{eq:vanishingMomentsQR}
%%%%%%%%%%%%%%%%%%%%%%%%%%%%%%%%%%%%%%%%%%%%%%%%%%%%%%%%%%%%%%%%%%%%%%%%%%%%%%%%
  \big[{\bs M}_{j,\Phi}^{\nu}, {\bs M}_{j,\Psi}^{\nu}\big]
= {\bs M}_{j+1}^{\nu} [{\bs Q}_{j,\Phi}^{\nu} , {\bs Q}_{j,\Psi}^{\nu} ]
  = {\bs R}^\intercal,
\end{equation}
cp.\ \eqref{eq:refinementRelation} and \eqref{eq:momentMatrix}.
Since ${\bs R}^\intercal$ is a lower triangular matrix, the first $k-1$ 
entries in its $k$-th column are zero. This corresponds to 
$k-1$ vanishing moments for the $k$-th functional generated 
by the transformation
${\bs Q}_{j}^{\nu}=[{\bs Q}_{j,\Phi}^{\nu} , {\bs Q}_{j,\Psi}^{\nu} ]$. 
By defining the first $m_{\Pcal}$ functionals as scaling functionals and 
the remaining ones as samplets, we obtain samplets that are orthogonal to
the primitives in \(\Pcal\). To guarantee that each leaf cluster
\(\nu_{\text{leaf}}\in\Tcal\) contains samplets, the condition $
|\nu_{\text{leaf}}|>m_{\Pcal}$ has to be met.

For leaf clusters, we define the functionals 
${\bs f}_N=\{f_1,\ldots,f_N\}$ as scaling functionals, i.e.,
$\mathbf{\Phi}_J^{\nu}\isdef\{f_i : i\in\nu \}$.
The scaling functionals of all clusters on a specific level $j$ 
then generate the spaces
\begin{equation}\label{eq:Vspaces}
%%%%%%%%%%%%%%%%%%%%%%%%%%%%%%%%%%%%%%%%%%%%%%%%%%%%%%%%%%%%%%%%%%%%%%%%%%%%%%%%
	\Xcal'_{j}\isdef \finspan{\varphi_{j,k}^{\nu} : k\in I_j^{\Phi,\nu}
	,\ j_\nu=j},
\end{equation}
while the samplets span the detail spaces
\begin{equation}\label{eq:Wspaces}
%%%%%%%%%%%%%%%%%%%%%%%%%%%%%%%%%%%%%%%%%%%%%%%%%%%%%%%%%%%%%%%%%%%%%%%%%%%%%%%%
\Scal'_{j}\isdef
	\finspan{\psi_{j,k}^{\nu} : k\in I_j^{\Psi,\nu},\ 
	j_\nu=j}.
\end{equation}
Herein, \(I_j^{\Phi,\nu}\) and \(I_j^{\Psi,\nu}\) are suitably chosen
index sets.
In particular, there holds \(\Xcal_{j+1}'=\Xcal_{j}'\oplus\Scal_{j}'\)
and, in case that the functionals
\(f_1,\ldots, f_N\) are linearly independent, \(\Xcal_{j}'\perp\Scal_{j}'\). 
The latter, as well as the fact that the samplets form an orthonormal 
basis is then a direct consequence of 
Lemma~\ref{lem:basisTrafo}. Therein, the matrix \(\bs U\) is obtained
by collecting all expansion coefficients from \eqref{eq:refinementRelation}
belonging to a particular samplet in the corresponding row of \(\bs U\).
The orthogonality of \(\bs U\) follows from
the images in \(\Rbb^N\) of the functionals of different 
clusters on the same level being orthogonal and
the orthogonality of the matrices \({\bs Q}_j^\nu\).
Finally, we remark that, for balanced binary trees \(\Tcal\), the samplet
basis can be constructed
with cost \(\Ocal(N)\), cp.\ \cite{HM22} and the references therein.

Avoiding the technical details concerning the
enumeration of samplets with respect to the cluster tree, 
we simply use the notation of Lemma~\ref{lem:basisTrafo} here.
According to \eqref{eq:matrcoeff1}, the samplets obtained by this particular 
choice of the isometry \({\bs U}\)
satisfy
\begin{equation}\label{eq:vanMom}
  (\psi_{i},p)_{\Bcal'\times\Bcal}=({\bs u}_{i},T_N^\star p)_{\ell^2}
  =0\quad\text{for all }p\in\Pcal,
\end{equation}
while for the dual samplets holds
\[
\big\langle\widetilde{\psi}_{i},p\big\rangle_{\Xcal_N}=0
\quad\text{for all }p\in\Pcal.
\]
Especially, in case when \({\bs f}_N\subset{\bs f}\) for a Riesz basis 
\({\bs f}\) in \(\Bcal'\), we even have that
\[
\big\langle\widetilde{\psi}_{i},p\big\rangle_{\Bcal}=0
\quad\text{for all }p\in\Pcal,
\]
as well as
\[
\langle\psi_{i},\psi_{i'}\rangle_{\Bcal'}=
\big\langle\widetilde{\psi}_{i},\widetilde{\psi}_{i'}
\big\rangle_{\Bcal}
=
\big(\psi_{i'},\widetilde{\psi}_{i'}\big)_{\Bcal'\times\Bcal}
=\delta_{i,i'}.
\]

\subsection{Properties of generalized samplets}
We give an abstract localization result for the samplet 
coefficients with respect to the set of primitives \(\Pcal\subset\Bcal\) and
the support of the particular samplet under consideration.
\begin{theorem} There holds for any \(v\in\Bcal\) that
  \[
    |(\psi_{i}, v)_{\Bcal'\times\Bcal}|\leq\sqrt{B_N}
    \inf_{p\in\Pcal}\|v-p\|_{\Bcal,\support{\psi_{i}}}
\]
with the semi-norm
\(\|v\|_{\Bcal,K}\isdef\inf
\{\|\tilde{v}\|_{\Bcal}: {\tilde{v}\in\Bcal\text{ and }\tilde{v}|_K=v}\}\)
for any compact set \(K\subset\Omega\) and the Bessel bound \(B_N\).
\end{theorem}
\begin{proof}
  Since \(\Bcal'\subset\Ecal'(\Omega)\), the samplet \(\psi_{i}\) has compact 
  support \(\support{\psi_{i}}\). Let \(\chi\in C^\infty_0(\Omega)\) 
be any function with \(\chi=1\) in a neighborhood of \(\support{\psi_{i}}\). 
Then, there holds 
\((\psi_{i}, v)_{\Bcal'\times\Bcal}=(\psi_{i},\chi v)_{\Bcal'\times\Bcal}\),
cp.\ \cite[Theorem 6.24]{Rud91}.
Further, we have due to the vanishing moment property \eqref{eq:vanMom}, 
for any \(p\in\Pcal\) that
\begin{align*}
  |(\psi_{i}, v)_{\Bcal'\times\Bcal}|&=|(\psi_{i}, v-p)_{\Bcal'\times\Bcal}|
 =\big|\big(\psi_{i},\chi(v-p)\big)_{\Bcal'\times\Bcal}\big|\\
  &=\big|\big({\bs u}_{i},T_N^\star\chi(v-p)\big)_{\ell^2}\big|
 \leq\|T_N^\star\chi(v-p)\|_{\ell^2}\leq\sqrt{B_N}\|\chi(v-p)\|_{\Bcal},
\end{align*}
where we used the Cauchy-Schwarz inequality in the second last step
and the continuity of the finite analysis operator in the last step.
Since \(p\in\Pcal\) and \(\chi\in C^\infty_0(\Omega)\)
can be chosen arbitrarily, with the constraint that
\(\chi=1\) in a neighborhood of \(\support{\psi_{i}}\), we arrive at
\[
  |(\psi_{i}, v)_{\Bcal'\times\Bcal}|       
  \leq\sqrt{B_N}\inf_{p\in\Pcal}\|v-p\|_{\Bcal,\support{\psi_{i}}},
\]
as claimed.
\end{proof}

\begin{remark}
The previous theorem can be translated into certain properties in the
coefficient space \(\mathbb{R}^N\). There, the localization of the approximation 
with respect to the samplet's support corresponds to a localization 
of coefficients with respect to the support of the coefficient sequence
\({\bs u}_{i}\). More precisely, there holds
  \begin{align*}
    |(\psi_{i},v)_{\Bcal'\times\Bcal}|
    &=|({\bs u}_{i},T_N^\star v)_{\ell^2}|
  =\big|\big({\bs u}_{i},T_N^\star (v-p)\big)_{\ell^2}\big|\\
    &=\Big|\Big({\bs u}_{i},
  T_N^\star (v-p)\big|_{\support{{\bs u}_{i}}}\Big)_{\ell^2}\Big|
  \leq\|T_N^\star (v-p)\|_{\ell^2,\support{{\bs u}_{i}}}
  \end{align*}
by the Cauchy-Schwarz inequality.
Herein, the last term amounts to a localized version of the norm corresponding
to the \(\langle\cdot,\cdot\rangle_{\Bcal}\)-inner product. 
\end{remark}
Even though samplets can be constructed with respect to arbitrary primitives,
the canonical choice is \(\Pcal=\Pcal_q\) being the space of all
polynomials of degree less or equal than \(q\). In this case, it is well
known that \(m_{\Pcal}=\binom{d+q}{q}\). For polynomial primitives,
the Bramble-Hilbert lemma, cf.\ \cite{BH70}, directly yields a localization
result akin to those for classical wavelets in Sobolev spaces.

\begin{cor.}\label{cor:decay}
%%%%%%%%%%%%%%%%%%%%%%%%%%%%%%%%%%%%%%%%%%%%%%%%%%%%%%%%%%%%%%%%%%%%%%%%%%%%%%%%
  Let $v\in W^{k,p}(\Omega)$, i.e., \(\Bcal= W^{k,p}(\Omega)\) and assume
  \(v\in W^{m,p}(O)\) with \(m\geq k\) for an open set 
  \(O\supset\support{\psi_{i}}\) and let \(q\geq m - 1\).
Then, there holds
\begin{equation}\label{eq:decaySobolev}
%%%%%%%%%%%%%%%%%%%%%%%%%%%%%%%%%%%%%%%%%%%%%%%%%%%%%%%%%%%%%%%%%%%%%%%%%%%%%%%%
 \big|(\psi_{i},v)_{\Omega}\big|\le\sqrt{B_N} C(m,O) 
 \big(\diam\support{\psi_{i}}\big)^{m-k}\|v\|_{W^{m,p}(O)}
\end{equation}
for some constant \(C(m,O)\) depending on \(m\) and \(O\) and the
Bessel bound \(B_N\).
\end{cor.}

%%%%%%%%%%%%%%%%%%%%%%%%%%%%%%%%%%%%%%%%%%%%%%%%%%%%%%%%%%%%%%%%%%%%%%%%%%%%%%%%
\subsection{Examples}\label{sec:exampfin0}
%%%%%%%%%%%%%%%%%%%%%%%%%%%%%%%%%%%%%%%%%%%%%%%%%%%%%%%%%%%%%%%%%%%%%%%%%%%%%%%%
The present section is devoted to examples, which connect the presented
construction to existing ones in literature.

\begin{example}\label{ex:RKHS} Let \((\Hcal,\langle\cdot,\cdot\rangle_\Hcal)\)
be a reproducing kernel Hilbert space of functions \(h\colon\Omega\to\Rbb\)
on the domain \(\Omega\subset\Rbb^d\) with a positive definite
reproducing kernel $\kappa_{\bs x}({\bs y})=K({\bs x},{\bs y})$. Given
a set 
\[
X_N\isdef\{{\bs x}_1,\ldots,{\bs x}_N\}\subset\Omega\]
of mutually 
distinct points, we consider the subspace 
\[\Xcal_N'\isdef\finspan{\delta_{{\bs x}_1},\ldots, \delta_{{\bs x}_N}}
\subset\Hcal'.
\]
As \(\Hcal\) is a Hilbert space, we of course have that $\Hcal\simeq\Hcal'$
by the Riesz isometry. Even so, we want to distinguish between those spaces, 
and consider $\delta_{\bs x}$ as elements in $\Hcal'$, and the Riesz representers
$\kappa_{\bs x}$ as elements in $\Hcal$. 
The finite frame operator \(S_N\colon \Hcal \rightarrow \Hcal'\) associated
to the Dirac-$\delta$-distributions spanning \(\Xcal_N\) is given by
\[
S_N h=\sum_{i=1}^N h({\bs x}_i)\delta_{{\bs x}_i}.
\]
Further, we have by the Riesz isometry that 
\(S_N^{-1}\colon\Xcal_N'\to\Xcal_N=\finspan{\kappa_{{\bs x}_1},\ldots,
\kappa_{{\bs x}_N}}\)
and there holds
\[
S_N \left( \sum
\limits_{i=1}^N \alpha_i \kappa_{{\bs x}_i} \right) 
= \sum \limits_{i=1}^N [{\bs K}{\bs\alpha}]_i\delta_{{\bs x}_i}
\] 
with ${\bs K}\isdef[K({\bs x}_i,{\bs
x}_j)]_{i,j=1}^N$ being the kernel matrix. By assumption ${\bs K}$ 
is an invertible matrix. Therefore 
\[
S_N^{-1}\delta_{{\bs x}_j}=\sum_{i=1}^N[{\bs K}^{-1}]_{i,j}\kappa_{\bs x_i}
\defis\tilde{\kappa}_{{\bs x}_j},
\]
is the dual basis satisfying 
\((\delta_{{\bs x}_j},\delta_{{\bs x}_k})_{\Xcal_N'}
=(\delta_{{\bs x}_j},\tilde{\kappa}_{{\bs x}_k})_{\Hcal'\times\Hcal}
=\delta_{j,k}\). 
Furthermore,
if \({\bs U}\) is the transformation matrix of the 
samplet transform, then the basis
\[
  \widetilde{\bs\Psi}={\bs U}^\intercal
  [\tilde{\kappa}_{{\bs x}_1},\ldots,\tilde{\kappa}_{{\bs
  x}_n}]^\intercal
\]
is exactly the dual embedded samplet basis introduced in \cite{BHM24}.
\end{example}

%\begin{example}
%Let \(\mathfrak{F} = \{f_i\}_i\) be a %n $\ell^2$-
%Riesz basis in \(\Bcal'\).
%The space \(\mathbb{R}^N = \ell^2 (N) \) is a reproducing kernel Hilbert 
%space with reproducing kernel 
%\[
%\kappa_i(j) = K (i,j)\isdef\delta_{i,j}=\begin{cases}1,& i=j\\ 0,& \text{else.}
%\end{cases}
%\]
%Using the corresponding reproducing property and the
%definition \eqref{eq:InnerPB} of the inner product on \(\Bcal\), there holds
%\begin{equation}
%(f_i,u)_{\Bcal'\times\Bcal} = \left(T_N^\star u\right)_i 
%=\big(\kappa_i,T_N^{\star}u\big)_{\ell^2}
%=\big(T_N \kappa_i,u\big)_{\Bcal'\times\Bcal}.
%\end{equation}
%This idea generalizes the reproducing kernel Hilbert space concept to more 
%general concepts
%than point evaluations. \textcolor{blue}{@ Michael: Bitte mehr erkl\"aren 
%.... was bringt das?}

%\textcolor{blue}{
%$T_N \kappa_i = \sum \limits_{j = 1}^N \delta_{i,j} f_j = f_i$. 
%Was haben wir gewonnen? 
%}
%\end{example}
\begin{example} Let 
  \(\Xcal_N=\finspan{\varphi_1,\ldots,\varphi_N}\subset H^1_0(\Omega)\),
for the domain \(\Omega\subset\Rbb^d\). We assume that the functions
\(\varphi_i\), are linearly independent such that \(\dim \Xcal_N=N\).
For example, in the context of finite element discretizations of partial 
differential equations, the basis elements \(\varphi_i\) are usually chosen as 
piecewise polynomial functions with respect to an underlying mesh and
exhibit local supports.
We introduce the
functionals \(\Xcal_N'=\finspan{f_1,\ldots,f_N}\subset H^{-1}(\Omega)\)  
by employing the Riesz isometry
\[
f_i(v)\isdef(\varphi_i,v)_{\Omega}=\int_\Omega \varphi_i v\d{\bs x},
\]
resulting in the isomorphy \(\Xcal_N\simeq\Xcal_N'\).
The finite frame operator 
$S_N\colon  H^1_0(\Omega) \rightarrow  H^{-1}(\Omega)$ is bijective 
from \(\Xcal_N\) to \(\Xcal_N'\) and given by
\[
  S_N \left(\sum \limits_{j=1}^N c_j \varphi_j \right) 
  = \sum_{i=1}^N\left(f_{i},\sum \limits_{j=1}^N c_j  \varphi_j
  \right)_{\Omega}f_{i} = \sum_{i=1}^N [{\bs M}{\bs c}]_i f_i  ,
\]
where ${\bs M}\isdef[(\varphi_i,\varphi_j)_{L^2}]_{i,j=1}^N$ is the Gramian
or, in the finite element context, the mass matrix, 
and ${\bs c} \isdef [c_i]_{i=1,\dots ,N}$ is the coefficient vector.
From the representation of \(S_N\), we directly infer the dual basis
\[
S_N^{-1}f_j=\tilde{f}_j=\sum_{i=1}^N[{\bs M}^{-1}]_{j,i}\varphi_{i},\quad
j=1,\ldots,N,
\]
satisfying 
\((f_j,f_k)_{\Xcal_N'}
=(f_j,\tilde{f_k})_{\Omega}=\delta_{j,k}\). 

Constructing samplets exhibiting
vanishing moments with respect to polynomials, results in wavelet-like functions,
akin to the Tausch-White construction, cf.\ \cite{TW03},
which can be constructed in a black-box fashion 
on arbitrary meshes.
\end{example}
\begin{example} \label{ex:3} As a closely connected setting, we finally consider 
the framework of operator adapted wavelets, see \cite{OS19}. Let
\(\Lcal\colon H^s_0(\Omega)\to H^{-s}(\Omega)\), \(s>0\), 
for some domain \(\Omega\subset\Rbb^d\), be a uniformly elliptic and continuous 
operator.
Consider the space \(\Xcal_N'=\finspan{f_1,\ldots,f_N}
\subset H^{-s}(\Omega)\). 
The operator \(\Lcal\) takes here the role of the Riesz isometry in the
previous two examples and we set
\(\Xcal_N=\finspan{\Lcal^{-1}f_1,\ldots,\Lcal^{-1}f_N}\). 
The finite frame operator 
$S_N\colon  H^s_0(\Omega) \rightarrow  H^{-s}(\Omega)$ is bijective 
from \(\Xcal_N\) to \(\Xcal_N'\) and given by
\[
  S_N \left(\sum \limits_{j=1}^N c_j\Lcal^{-1}f_j \right) 
  = \sum_{i=1}^N\left(f_{i},\sum \limits_{j=1}^N c_j\Lcal^{-1}f_j
  \right)_{\Omega}f_{i} = \sum_{i=1}^N [{\bs M}{\bs c}]_if_i,
\]
with the Gramian
\({\bs M}\isdef[(f_i,\Lcal^{-1}f_j)_{\Omega}]_{i,j=1}^N\),
which amounts to the stiffness matrix of the inverse operator \(\Lcal^{-1}\)
in the finite element context. As before, the canonical dual basis
is given by
\[
S_N^{-1}f_j=\tilde{f}_j=\sum_{i=1}^N[{\bs M}^{-1}]_{i,j}\Lcal^{-1}f_{i}
\]
and satisfies
\((f_j,f_k)_{\Xcal_N'}
=(f_j,\tilde{f}_k)_{\Omega}=\delta_{j,k}\). 

The elements \(\tilde{f}_i\in H^s_0(\Omega), \ i=1,\ldots,N\) are called
\emph{optimal recovery splines}, cp.\ \cite{OS19}. 
\end{example}

\section*{Acknowledgments}
We thank Nicki Holighaus for fruitful discussions and helpful remarks.

The work of Peter Balazs has been supported by the FWF projects 
LoFT (P~34624) and NoMASP (P~34922). 

The work of Michael Multerer has been supported by the SNSF 
starting grant 
``Multiresolution methods for unstructured data'' (TMSGI2\_211684).

%%%%%%%%%%%%%%%%%%%%%%%%%%%%%%%%%%%%%%%%%%%%%%%%%%%%%%%%%%%%%%%%%%%%%%%%%%%%%%%%
\bibliographystyle{abbrv}
\bibliography{biblioall}
%%%%%%%%%%%%%%%%%%%%%%%%%%%%%%%%%%%%%%%%%%%%%%%%%%%%%%%%%%%%%%%%%%%%%%%%%%%%%%%%
\end{document}